\documentclass[9pt]{article}
\usepackage{amsfonts, amsmath, amsthm, amstext, amssymb, amscd}
\usepackage[dvips]{color}
\usepackage{graphicx}
\usepackage[all]{xy}
\usepackage{color}
\usepackage{ifpdf}
\input epsf
\newtheoremstyle{thm}{0.32em}{0.32ex}{\it}{}{\bf}{.}{0.5em}{\thmname{#1}\thmnumber{ #2}\thmnote{ ({\rm #3})}}
\theoremstyle{thm}
\newtheorem{theorem}{Theorem}[section]

\newtheorem{lemma}[theorem]{Lemma}

\newtheorem{proposition}[theorem]{Proposition}

\newenvironment{Acknowledgements}{\noindent{\bf Acknowledgements:}}{\vspace{3mm}}
\newcounter{chap}
\newcounter{sect}

\usepackage{titlesec}
\usepackage[normalem]{ulem}
\titleformat{\section}
{\normalfont\large\bfseries}{\thesection}{1em}{}
\titlespacing{\section}{0cm}{0.6cm}{0.2cm}

\makeatletter
\def\enumerate{%
 \ifnum \@enumdepth >\thr@@\@toodeep\else
   \advance\@enumdepth\@ne
   \edef\@enumctr{enum\romannumeral\the\@enumdepth}%
     \expandafter
     \list
       \csname label\@enumctr\endcsname
       {\usecounter\@enumctr\def\makelabel##1{\hss\llap{##1}}%
         \addtolength{\parsep}{1.56pt}
         \addtolength{\listparindent}{0pt} 
         \addtolength{\itemsep}{-8pt} 
         \addtolength{\topsep}{-1pt}} 
 \fi}
\makeatother

\def\AA{\mathbb{A}}
\def\CC{\mathbb{C}}

\def\HH{\mathbb{H}}
\def\NN{\mathbb{N}}
\def\PP{\mathbb{P}}
\def\QQ{\mathbb{Q}}
\def\RR{\mathbb{R}}

\def\ZZ{\mathbb{Z}}

\renewcommand{\baselinestretch}{1.2}

\renewenvironment{proof}{\noindent{\em
Proof.}\hspace{0.3em}}{\hfill\qed\vspace{0ex}}

\newenvironment{proofof}[1]{\noindent{\em #1.}\hspace*{0.3em}}{\hfill\qed\vspace{0ex}}

\begin{document}

\renewcommand{\baselinestretch}{1.2}

\begin{center}
{\Large\bf Isogeny orbits in a family of abelian varieties}\\
{\sc By Qian Lin and Ming-Xi Wang}\\
{\it Harvard University, University of Salzburg\\}

\quad E-mail address:  qianlin88@gmail.com, Mingxi.Wang@sbg.ac.at
\end{center}

\begin{abstract}
We prove that if a curve of a non-isotrivial family of abelian
varieties over a curve contains infinitely many isogeny orbits of a
finitely generated subgroup of a simple abelian variety then it is
special. This result fits into the context of Zilber-Pink conjecture
and partially generalizes a result of Faltings. Moreover by using
the polyhedral reduction theory we give a new proof of a result of
Bertrand.
\end{abstract}

\begin{figure}[b]
\rule[-2.5truemm]{5cm}{0.1truemm}\\[1mm]
{\footnotesize {\it AMS Classification (2010):} Primary, 11G18;
Secondary 14K12, 11G50.
\par Ming-Xi Wang was partially supported by Austrian Science Fund(FWF): P24574. }
 \end{figure}

\section{Introduction}
In this paper we are interested in a curve of an abelian scheme that
contains infinitely many isogeny orbits of a finitely generated
group of a simple abelian variety. We prove that it is special.
Zilber-Pink have conjectured, roughly speaking, that a subvariety
containing many special points must be special. This generalizes
many well-known problems including conjectures of Mordell-Lang,
Manin-Mumford and Andr\'e-Oort. Special points considered in this
paper are isogeny orbits which are closely related to the
generalized Hecke orbits considered by Zilber-Pink. Therefore our
result fits into the context of their conjectures.

Let $S$ be a smooth irreducible abstract algebraic curve over
$\overline{\QQ}$, and $\pi: A \to S$ be an abelian scheme. An
abelian scheme $A \to S$ refers to a smooth proper group scheme with
geometrically connected fibers. Then $A$ can be regarded as a smooth
family of abelian varieties over $S$. Take an abelian variety $A'$
defined over $\overline{\QQ}$ and a finitely generated group $\Gamma
\subset A'(\overline{\QQ})$. We call a point $q \in
A_t(\overline{\QQ})$, where $t \in S(\overline{\QQ})$, special if
there exist an isogeny $\phi: A' \to A_t$ and $\gamma \in \Gamma$
with $\phi(\gamma)=q$. In this paper we prove

\begin{theorem}\label{main theorem}
Assume that $A'$ is simple and $A$ is non-isotrivial.  If an
irreducible Zariski closed algebraic curve $X$ of $A$, over
$\overline{\QQ}$, contains infinitely many special points, then
either $X$ is a subtorus of $A_q$ for some $q \in S(\overline{\QQ})$
or there exists $n \in \NN $ such that $[n]X(\overline{\QQ})=0$.
\end{theorem}

In the next section we prove a partially stronger result in case of
a family of elliptic curves, and indicate two major obstructions
preventing that argument from working in case of families of abelian
varieties. In Section 3 we use the polyhedral reduction theory to
give a new proof of a result of Bertrand, which is crucial for this
paper. In Section 4 we present the proof of the main theorem. The
basic strategy of this paper is to compare different heights in
number theory, including geometric Faltings height, Neron-Tate
height and Weil height.

Throughout this paper $h_{F}(A)$ refers to the geometric Faltings
height of an abelian variety $A$ over $\overline{\QQ}$, $h_{X, D}:
X(\overline{\QQ}) \to \RR$ refers to a Weil height function of a
variety over $\overline{\QQ}$ with respect to the divisor $D$ and
$\hat{h}_{A, D} : A(\overline{\QQ}) \to \RR_{\ge 0}$ refers to the
canonical height function of an abelian variety over
$\overline{\QQ}$ with respect to a symmetric divisor $D$. Two
linearly equivalent divisors respectively isomorphic line bundles
are connected by $\sim$ respectively $\cong$. The set of linear
equivalence of divisors or line bundles of $X$ is $\mathrm{Pic}(X)$.
For a complete nonsingular curve $X$ over $\overline{\QQ}$, we have
a canonical surjective homomorphism $\deg: \mathrm{Pic}(X) \to \ZZ.$
Let $\mathcal{L}$ be an invertible sheaf of an abelian variety $A$,
$\chi(\mathcal{L})$ be its Euler characteristic and
$\lambda_{\mathcal{L}}$ be the morphism from $A(\overline{\QQ})$ to
$\mathrm{Pic}(A)$. The subgroup $\mathrm{Pic}^0(A)$ of
$\mathrm{Pic}(A)$ consists of invertible sheaves $\mathcal{L}$ for
which $\lambda_{\mathcal{L}} \equiv 0.$ The Neron-Severi group of
$A$ is denoted by $\mathrm{NS}(A)$. Points of the dual abelian or
Picard variety $A^{\vee}$ of $A$ parametrize the elements of
$\mathrm{Pic}^0(A)$. A homomorphism of abelian varieties $\phi: A
\to B$ gives rise to a dual $\phi^{\vee}: B^{\vee} \to A^{\vee}.$
Write $\mathrm{End}(A) = \mathrm{End}_{\overline{\QQ}}(A)$ and
$\mathrm{End}^0(A) = \mathrm{End}(A) \otimes \QQ,$ on which we have
another $\deg$ function. The set of isomorphism classes of pairs
$(A, \lambda)$ with $A$ an abelian variety of dimension $g$ and
$\lambda$ a polarization of $A$ of degree $d = \chi(\lambda)$ is
parameterized by the Siegel modular variety $M_{g, d}$, in
particular $M_{1, 1} = \AA^1.$ If $\Gamma$ is a finitely generated
abelian group then we let $\Gamma_t$ respectively $\Gamma_{nt}$ to
be the torsion subgroup respectively its complement.

\section{Isogeny orbits in a family of elliptic curves}
In this section we let $S$ be a smooth irreducible abstract
algebraic curve over $\overline{\QQ}$, and $\pi: A \to S$ an abelian
scheme of dimension one. Then $A$ can be regarded as a smooth family
of elliptic curves over $S$. Take an elliptic curve $A'$ defined
over $\overline{\QQ}$ and a $p \in A'(\overline{\QQ})$. We call $q
\in A_t(\overline{\QQ})$, where $t \in S(\overline{\QQ})$, special
if there exists either an isogeny $\phi: A' \to A_t$ with
$\phi(p)=q$, or an isogeny $\phi: A_t \to A'$ with $\phi(q) =p$. We
prove
\begin{proposition}\label{familiy of elliptic curves}
Assume that $A$ is non-isotrivial.  If an irreducible Zariski closed
algebraic curve $X$ of $A$, over $\overline{\QQ}$, contains
infinitely many special points, then either $X$ is some special
fiber $A_t$ that is isogenous to $A'$ or there exists $n \in \NN $
such that $[n]X(\overline{\QQ})=0$.
\end{proposition}
\begin{proof}
Firstly we assume that $X$ is not any fiber $A_t$, otherwise there
is nothing to prove. Secondly we notice that it suffices to prove
the result under the assumption that $X$ is a section $s: S \to A$
of $ \pi: A \to S$. Indeed in the general case if let $X'$ be a
smooth resolution of $X$, then $A \times_{S} X' \to X'$ is also a
smooth family of elliptic curves over $X'$ and we write $f: X' \to
A$ to be the natural morphism.
\begin{displaymath}
\xymatrix{
  X' \ar@/_/[ddr]_{f} \ar@/^/[drr]^{id}
    \ar@{.>}[dr]|-{s}                   \\
   & A \times_S X' \ar[d]^{pr_1} \ar[r]_{pr_2}
                      & X' \ar[d]_{\pi \circ f}    \\
   & A \ar[r]^{\pi }     & S               }
\end{displaymath}
The above commutative diagram provides a section $s: X' \to A
\times_S X'$ of $A \times_S X'/X'$. Moreover it is easy to check
that $s(X') \subset A \times_S X'$ contains infinitely special
points if and only if $X \subset A$ does and that $[n]X = 0$ if and
only if $([n] \circ s) X = 0$. Therefore the general case reduces to
a special case that $X$ comes from a section.

If $p$ is a torsion point and $A'$ has complex multiplication, then
our assertion is a special case of a result of Andr\'e \cite{Andre}.
If $p$ is a torsion point and $A'$ has no complex multiplication,
then by a lemma of Habegger \cite[Lemma 5.8]{Habegger} there are
only finitely many elliptic curves isogenous to $A'$ with bounded
height (The proof of Habegger's statement relies heavily on the work
of Szpiro and Ullmo). This makes Andr\'e's argument valid line by
line after replacing only Poonen's lemma \cite{Poonen} by Habegger's
lemma. From here on we assume that $p$ is not torsion.

Given an elliptic curve $E$ over $\overline{\QQ}$ we write
$\hat{h}_{E}: E(\overline{\QQ}) \to \RR_{\ge 0}$ for the canonical
height function with respect to the divisor given by the zero
section $(0)$ of $E/\overline{\QQ}$. For any other symmetric divisor
$D$ of $E$, the associated canonical height function satisfies
\begin{eqnarray*}
\hat{h}_{E,\ D} = \deg{D} \cdot \hat{h}_{E},
\end{eqnarray*}
as $D \sim \deg{D}(0)$ for any symmetric divisor $D$ of $E$. Before
proceeding we notice that if $\phi: E' \to E$ is an isogeny of
elliptic curves over $\overline{\QQ}$ with $\phi(a) = b$, where $a,
b$ are closed points, then we have
\begin{eqnarray*}
\hat{h}_{E}(b) = \hat{h}_{E',\ \phi^{*}(0)}(a) = \deg(\phi^{*}(0))
\hat{h}_{E'}(a) = \deg{\phi} \cdot \hat{h}_{E'}(a).
\end{eqnarray*}

Let $D$ be the divisor given by the zero section of the abelian
scheme $A/S$, then the canonical height function on
$A_t(\overline{\QQ})$ with respect to $\deg D_t$ is simply
$\hat{h}_{A_t}$ for any $t \in S(\overline{\QQ}).$ According to the
non-isotriviality of $A$, the modular map $j: S(\overline{\QQ}) \to
\AA^1(\overline{\QQ})$ is non-constant. Without ambiguity we write
$h: \AA^1(\overline{\QQ}) \to \RR$ for the standard Weil height
function on $\AA^1$ and $h: s \in S(\overline{\QQ}) \to h(j(s)) \in
\RR$ for the Weil height function on $S$ with respect to $j^*((0)).$

There are two types of special points. The back orbit of $p$ and the
forward orbit of $p$. We first assume that $X$ contains infinitely
many back orbits $q_i(i = 1, 2, \ldots)$ of $p$. If $q_i \in
A_{t_i}$, where $t_i \in S(\overline{\QQ})$, then $q_i = s(t_i)$.
Let $\phi_i: A_{t_i} \to A'$ be the isogeny that satisfies
$\phi_i(q_i)=p$ then we have
\begin{eqnarray}\label{relation of canonical heights}
\hat{h}_{A_{t_i}}(q_i) = \frac{\hat{h}_{A'}(p)}{\deg{\phi_i}}.
\end{eqnarray}
Using the lemma of Habegger \cite[Lemma 5.8]{Habegger}, there are
only finitely many elliptic curves over $\overline{\QQ}$ within the
isogeny class of $E'$ with bounded Weil height. Using the
non-isotriviality of $A$, given any elliptic curve $E_1$ there are
only finitely many $i$ such that $E_{t_i}$ is isomorphic to $E_1$
over $\overline{\QQ}$. These two facts clearly lead to
\begin{eqnarray}\label{goestoinfinity}
\lim_{i \to \infty} h(t_i) \to \infty.
\end{eqnarray} By Silverman's specialization theorem
\cite{Silverman} there is a constant $C$ such that
\begin{eqnarray}\label{Specialization theorem}
\lim_{h(t) \to \infty} \frac{\hat{h}_{A_t}(s(t))}{h(t)} = C.
\end{eqnarray}
Because of (\ref{goestoinfinity}) we can apply (\ref{Specialization
theorem}) to $t_i$ and obtain
\begin{eqnarray*}
\lim_{i \to \infty} \frac{\hat{h}_{A_t}(q_i)}{h(t_i)} = C.
\end{eqnarray*}
Using (\ref{relation of canonical heights}) we have
\begin{eqnarray*}
\lim_{i \to \infty} \frac{\hat{h}_{A'}(p)}{\deg{\phi_i} \cdot
h(t_i)} = C.
\end{eqnarray*}
As $p$ is not torsion, we have $h_{A'}(p) > 0$ and therefore the
above identity gives $C = 0$. Recall in Silverman's specialization
theorem \cite{Silverman}, the constant $C =0$ means that the
canonical height of $X$ regarded as a point in the abelian variety
$A_{\eta}$ over the generic point is zero. According to the
non-isotriviality of $A$, the $\overline{\QQ}(S)/\overline{\QQ}$
trace of $A$ is trivial. This implies that $X$ is a torsion of the
abelian variety over the generic fiber. There exists $n \in \NN$
such that $[n]X(\overline{\QQ}) = 0$, a contradiction to our
assumption that $p$ is not torsion.

Now we assume that $X$ contains infinitely many forward orbits
$\{q_i\}_{i=1}^{\infty}$ of $p$. Let $\phi_i: A' \to A_{t_i}$ be the
isogeny that satisfies $\phi_i(p)=q_i$, then
\begin{eqnarray}\label{relation of canonical heights2}
\hat{h}_{A_{t_i}}(q_i) = \deg{\phi_i} \cdot \hat{h}_{A'}(p).
\end{eqnarray}
The inequality of Faltings \cite{Fa83} implies that
\begin{eqnarray}\label{Faltings inequality}
h_F(A_{t_i}) \leq h_F(A') + \log{\deg{\phi_i}}/2.
\end{eqnarray}
We claim that $\lim_{i \to \infty} \deg{\phi_i} = \infty$. Otherwise
there are infinitely many $i$ such that $A_{t_i}$ are isomorphic to
each other over $\overline{\QQ}$, a contradiction to the fact that
$A$ is non-isotrivial. Therefore (\ref{relation of canonical
heights2}), (\ref{Faltings inequality}) and $\hat{h}_{A'}(p) > 0$
lead to
\begin{eqnarray*}
\lim_{i \to \infty} \frac{\hat{h}_{A_{t_i}}(s(t_i) = q_i)}{h(t_i)} =
\infty.
\end{eqnarray*}
The identity (\ref{goestoinfinity}) is still valid. This makes the
above equality a contradiction to Silverman's specialization theorem
(\ref{Specialization theorem}).
\end{proof}

The above argument does not work in the context of families of
abelian varieties for the following reasons. Firstly the relation
between canonical heights of points on isogenious abelian varieties
is not as simple as in (\ref{relation of canonical heights}).
Secondly the lemma of Habegger is not known in higher dimensional
case. More precisely we do not know whether within an isogeny class
of abelian varieties there are only finitely many ones with bounded
height. Without this result we have no validity of
(\ref{goestoinfinity}) in general yet, which is essential if we want
to directly apply Silverman's specialization theorem.

\section{Polyhedral reduction theory and canonical heights}
In this section we give a proof of Lemma \ref{main lemma} below,
based on the polyhedral reduction theory \cite{AMRT}. Actually this
lemma is not new. G. R\'emond pointed out to us that it is
equivalent to the main theorem of Bertrand \cite{Ber} in case of
simple abelian varieties, linked by the theorem of Mordell-Weil. We
still present the proof here, as our approach is rather distinct
from Bertrand's original one.

Throughout this section $A$ is an abelian variety over
$\overline{\QQ}$. Any symmetric line bundle $\mathcal{L}$ of $A$
defines a canonical height function $\hat{h}_{A, \mathcal{L}}$ that
is quadratic on $A(\overline{\QQ})$. We remark that $\hat{h}_{A,
\mathcal{L}}$ depends only on the class of $\mathcal{L}$ in
$\mathrm{NS}(A)$. Indeed if a symmetric line bundle $\mathcal{L}$
maps to zero in the short exact sequence
\begin{eqnarray*} 0 \rightarrow \mathrm{Pic}^0(A) \rightarrow
\mathrm{Pic}(A) \rightarrow \mathrm{NS}(A) \rightarrow 0,
\end{eqnarray*}
then we have $\mathcal{L} \in \mathrm{Pic}^0(A)$. This leads to
$[-1]^*\mathcal{L} = \mathcal{L}^{-1}$. Together with
$[-1]^*\mathcal{L} = \mathcal{L}$ we have $\mathcal{L}^2 = 0$ and
therefore $\hat{h}_{A, \mathcal{L}} \equiv 0.$ Hence there will be
no confusion when we label the canonical height function of a
symmetric line bundle to its image in the Neron-Severi group.

When $A$ is a simple abelian variety, a recent result of Kawaguchi
and Silverman \cite{KawaguchiSilverman} tells us that for any
nonzero nef symmetric $\mathcal{L} \in \mathrm{Pic}(A) \otimes \RR$
the canonical height $\hat{h}_{A, \mathcal{L}}(x) = 0$ if and only
if $x \in A(\overline{\QQ})_{t}$.

The endomorphism algebra $\mathrm{End}^0(A)$ is semi-simple and
contains $\mathrm{End}(A)$ as a lattice. The unit
$(\mathrm{End}^0(A) \otimes \RR)^{\times}$ is reductive, and
$\mathrm{Aut}(A)$ is an arithmetic group. The function $\deg$
extends to a homogeneous function of degree $2g$ on
$\mathrm{End}^0(A) \otimes \RR$.

For any line bundle $\mathcal{L}$ the theorem of square leads to a
group homomorphism
\begin{eqnarray*}
\lambda_{\mathcal{L}}: A(\overline{\QQ}) \to
A^{\vee}(\overline{\QQ})
\end{eqnarray*}
which takes $x$ to $T_x^*\mathcal{L} \otimes \mathcal{L}^{-1}$. It
is an isogeny if and only if $\mathcal{L}$ is ample.

From here on we fix an ample line bundle  $\mathcal{N}$, which
defines a Rosati involution $\dagger$ of $ \phi \in
\mathrm{End}^0(A)$ by $\phi^{\dagger} = \lambda_{\mathcal{N}}^{-1}
\circ \phi^{\vee} \circ \lambda_{\mathcal{N}}$. The map
\begin{eqnarray*}
\mathrm{NS}(A)_{\QQ} \to \mathrm{End}^0(A)
\end{eqnarray*}
defined by $\mathcal{L} \mapsto \lambda_{N}^{-1} \circ
\lambda_{\mathcal{L}}$ identifies $\mathrm{NS}(A)_{\QQ}$ with the
subset of $\mathrm{End}^0(A)$ of elements fixed by $\dagger$. Given
$\phi \in \mathrm{Aut}(A)$ and $\mathcal{L} \in \mathrm{Pic}(A)$ it
is straightforward to check that $\lambda_{\phi^*(\mathcal{L})} =
\phi^{\vee} \circ \lambda_{\mathcal{L}} \circ \phi$. Which extends
to be an action of $\mathrm{End}^0(A)$ on $\mathrm{NS}(A)_{\QQ}
\subset \mathrm{End}^0(A)$ by $\alpha^{\phi} = \phi^{\dagger} \circ
\alpha \circ \phi.$ The bilinear form
\begin{eqnarray*}
\langle \phi, \psi \rangle \mapsto \mathrm{Tr}(\phi \circ
\psi^{\dagger}).
\end{eqnarray*}
on $\mathrm{End}^0(A) \times \mathrm{End}^0(A)$ is positive
definite.

As a finite dimensional algebra over $\RR$ with a positive
involution, $\mathrm{End}^0(A) \otimes \RR$ is isomorphic to
$\prod_{i}M_{r_i}(\RR) \times \prod_{j}M_{s_j}(\CC) \times
\prod_{k}M_{t_k}(\HH)$ where the involution is given by
conjugations. Under this identification $\mathrm{N}^1(A) =
\mathrm{NS}(A) \otimes \RR$ respectively the ample cone
$\mathrm{Amp}(A)$ is isomorphic to $\prod_{i}H_{r_i}(\RR) \times
\prod_{j}H_{s_j}(\CC) \times \prod_{k}H_{t_k}(\HH)$ respectively
$\prod_{i}P_{r_i}(\RR) \times \prod_{j}P_{s_j}(\CC) \times
\prod_{k}P_{t_k}(\HH)$, where $H_r$ is the space of symmetric or
Hermitian symmetric matrices and $P_r$ consists of positive ones.
Let $G(\mathrm{Amp}(A))$ be the automorphism group of the cone
$\mathrm{Amp}(A)$ and $G(\mathrm{Amp}(A))^0$ its identity component.
The homomorphism $(\mathrm{End}^0(A) \otimes \RR)^{\times} \to
G(\mathrm{Amp}(A))^0$ is surjective. According to Ash's main result
of polyhedral reduction theory, there exists a rational polyhedral
cone $F \subset \overline{\mathrm{Amp}(A)}$ such that
$(\mathrm{Aut}(A) \cdot F) \cap \mathrm{Amp}(A) = \mathrm{Amp}(A).$
For more details of this paragraph we refer \cite{AMRT} and
\cite{Prendergast}.

Using the theorem of Ash we give a new proof of
\begin{lemma}[Bertrand]\label{main lemma}
Let $A$ be a simple abelian variety defined over $\overline{\QQ}$,
and let $\Gamma \subset A(\overline{\QQ})$ be a finitely generated
group. Then there exists a constant $C > 0$ depending on $\Gamma$
and $A$ such that for any symmetric ample divisor $\mathcal{M} \in
\mathrm{Pic}(A)$ and nontorsion $x \in \Gamma$ we have
\begin{eqnarray*}
\hat{h}_{A, \mathcal{M}}(x) \ge C (\chi(\mathcal{M}))^{1/g}.
\end{eqnarray*}
\end{lemma}
\begin{proof}
It is well-known that $\mathrm{End}(A)$ is of finite rank, hence
$\mathrm{End}(A)(\Gamma)$ is also a finitely generated group.
Therefore it suffices to prove our lemma under the assumption that
$\Gamma$ is invariant under $\mathrm{Aut}(A).$

By the polyhedral reduction theory \cite{AMRT}, there is a rational
polyhedral fundamental domain $F \subset \overline{\mathrm{Amp}(A)}$
under the action of $\mathrm{Aut}(A).$ The rationality of $F$
guarantees that there is a basis $\{v_1, \ldots, v_t \} \subset
\mathrm{NS}(A) \cap \overline{\mathrm{Amp}(A)}$ such that if $w \in
F \cap \mathrm{NS}(A)$ then there are nonnegative real numbers $r_i$
with $w = \sum_{i=1}^t r_iv_i.$

We have $\Gamma = \Gamma_{t} + \Gamma_{nt}$. Because $v_i$ are nef,
a result of Kawaguchi-Silverman \cite{KawaguchiSilverman} tells us
that $\hat{h}_{A, v_i}$ is a positive bilinear function on the
finitely generated abelian $\Gamma_{nt},$ therefore there exists a
constant $c_1$ such that
\begin{eqnarray}\label{basic nerontate theory}
\hat{h}_{A, v_i}(\gamma) \ge c_1
\end{eqnarray}
for all $1 \leq i \leq t$ and all non-zero $\gamma \in \Gamma_{nt}$.
Furthermore for any $\mathcal{M} \in \mathrm{Amp}(A)$ and $x = x_1
+x_2 \in \Gamma$ we have $\hat{h}_{A, \mathcal{M}}(x) = \hat{h}_{A,
\mathcal{M}}(x_2).$ Consequently the inequality (\ref{basic
nerontate theory}) is valid for all $1 \leq i \leq t$ and
non-torsion $\gamma \in \Gamma$.

Take a symmetric ample devisor with image $\mathcal{M} \subset F$,
then we have nonnegative real numbers $r_i$ such that $\mathcal{M} =
\sum_{i=1}^t r_iv_i$. In particular for any non-torsion $\gamma \in
\Gamma$ we have
\begin{eqnarray}\label{111}
\hat{h}_{A, \mathcal{M}}(\gamma) \ge c_1 \max_{i=1}^t\{r_i\}
\end{eqnarray}
The degree function $\deg: \mathrm{End}^0(A) \otimes \RR \to \RR$ is
homogeneous of degree $2g$. Take $\deg\mathcal{M}$ to be the degree
of the image of $\mathcal{M}$ in $\mathrm{End}^0(A) \otimes \RR$.
Then we shall have
\begin{eqnarray}\label{222}
\deg \mathcal{M} =
\deg(\lambda_{\mathcal{M}})/\deg(\lambda_{\mathcal{N}}) = c_2
\chi(\mathcal{M})^2
\end{eqnarray}
where $c_2 = 1/\deg(\lambda_{\mathcal{N}}).$ The homogeneity of
$\deg: \mathrm{End}^0(A) \otimes \RR \to \RR$ implies that there is
a positive constant $c_3$ which depends only on $v_i$ but not on
$\mathcal{M}$ such that
\begin{eqnarray}\label{333}
\deg{\mathcal{M}} \leq  c_3\sum_{i=1}^t r_i^{2g}.
\end{eqnarray}
The above (\ref{111}), (\ref{222}) and (\ref{333}) obvious lead to a
constant $C
> 0$ such that
\begin{eqnarray*}
\hat{h}_{A, \mathcal{M}}(\gamma) \ge C (\chi{(\mathcal{M})})^{1/g}
\end{eqnarray*}
for all $\mathcal{M} \subset F \cap \mathrm{Amp}(A)$ and non-torsion
$\gamma \in \Gamma$.

For general $\mathcal{M} \in \mathrm{Amp}(A)$, there exists $\sigma
\in \mathrm{Aut}(A)$ such that $\sigma^{*}(\mathcal{M}) \in F$. We
have assumed that $\sigma^{-1}(x) \in \Gamma$ and it is clear that
$\chi{(\sigma^{*}(\mathcal{M}))} = \chi{(\mathcal{M})}$. Therefore
for any non-torsion $\gamma \in \Gamma$
\begin{eqnarray*}
\hat{h}_{A, \ \mathcal{M}}(\gamma) = \hat{h}_{A,\
\sigma^{*}(\mathcal{M})}(\sigma^{-1}(\gamma)) \ge C
(\chi{\mathcal{(M)}})^{1/g},
\end{eqnarray*}
which proves what has been claimed in the lemma.
\end{proof}

\section{Proof of the main theorem}
It is unknown to us whether in an isogeny class of abelian varieties
there are only finitely many ones with bounded height. Therefore we
can not directly use Silverman's specialization theorem as before.
Instead we shall use some arguments of \cite{Silverman} to prove our
main theorem.

\begin{proofof}{Proof of Theorem \ref{main theorem}}
Firstly we may assume that $\Gamma$ is invariant under
$\mathrm{Aut}(A')$. Secondly by the same trick used in the proof of
Proposition \ref{familiy of elliptic curves} we assume that $X$ is a
section $s: S \to A$ of $\pi: A \to S$. We write $\epsilon: S \to A$
for the zero section.

In the decomposition $\Gamma = \Gamma_{t} + \Gamma_{nt}$,
$\Gamma_{t}$ is finite, there exists a positive integer $n$ such
that $[n] \Gamma_{t} =0.$ If there are infinitely many $t \in
S(\overline{\QQ})$ such that there exists $\gamma_t \in \Gamma_{t}$
and isogeny $\phi_t: A' \to A_t$ with $\phi_t(x_t) = s(t)$, then we
also have $[n] s(t) =0$. This implies that $[n]X(\overline{\QQ})$
intersects with the zero section infinitely many times. This leads
to $[n]X(\overline{\QQ}) = 0$.

Now we assume that there are infinitely many distinct $t_i \in
S(\overline{\QQ}) (i=1, 2, 3, \ldots )$ such that exists $\gamma_i
\in \Gamma \setminus \Gamma_{t}$ and isogenies $\phi_i: A' \to
A_{t_i}$ with $\phi_i(\gamma_i) = s(t_i)$.

By theory of Theta functions, there is a smooth semiabelian scheme
$\overline{A}/\overline{S}$ extending the family $A/S$ (where
$\overline{S}$ is the smooth compactification of $S$) together with
a symmetric very ample line bundle $\mathcal{L}'$ of $A/S$ that
extends to a very ample one $\mathcal{L}$ of
$\overline{A}/\overline{S}$. Moreover $\mathcal{L}$ is indeed very
ample on $\overline{A}/\overline{\QQ}$ and makes the total space a
projective variety. For details of this paragraph we refer
\cite{Namikawa}.

The Euler characteristic $\chi(\mathcal{L}_t)$ is a constant
function of $t \in S(\overline{\QQ})$, and we denote it by $d$.
Because $A$ is non-isotrivial, the modular map $j: S \to M_{g, d}$
is not constant. We claim that $$\lim_{i \to \infty} \deg{\phi_i} =
\infty.$$ Otherwise there are infinitely many $t \in
S(\overline{\QQ})$ such that $A_t$ are all isomorphic to each other
over $\overline{\QQ}$. According to a well-known geometric
finiteness theorem, given any abelian variety $A^0$ and $d \in \NN$
there are only finitely many isomorphism classes of polarized
abelian varieties $(A^0, \lambda)$ with $\lambda$ of degree $d$.
These two facts together force the modular map $j: S \to M_{g, d}$
to be constant. This contradicts to the non-isotriviality of $A$ and
proves the claim.

Concerning isogenous $\phi_i$ of abelian varieties Faltings'
inequality (\cite{Fa83}) gives
\begin{eqnarray}\label{general faltings inequality}
h_F(A_{t_i}) \leq h_{F}(A') + \log\deg{\phi_i}/2.
\end{eqnarray}
Under the isogeny canonical heights satisfy $\hat{h}_{A_{t_i},
\mathcal{L}_{t_i}}(s(t_i)) = \hat{h}_{A',
\phi_i^*(\mathcal{L}_{t_i})}(\gamma_i)$. The Euler characteristics
satisfy $\chi(\phi_i^*(\mathcal{L}_{t_i})) = \deg{\phi_i} \cdot
\chi(\mathcal{L}_{t_i}) =  d \deg{\phi_i}$. By the lemma in the last
section there exists a positive constant $c_1$ such that
\begin{eqnarray}\label{change of canonical height}
\hat{h}_{A_{t_i}, \, \mathcal{L}_{t_i}}(s(t_i)) \ge
c_1(\deg{\phi_i})^{1/g}.
\end{eqnarray}

The group morphism $[2]: A/S \to A/S$ extends to a morphism on the
semiabelian scheme $\overline{A}/\overline{S}$. As mentioned before
$\mathcal{L}$ gives a projective embedding of $\overline{A}$. By
this condition a theorem of Silverman-Tate \cite[Theorem
A]{Silverman} applies, and consequently there exists positive
constants $c_2$ and $c_3$ such that
\begin{eqnarray}\label{SilvermanTate}
\left|\,\hat{h}_{A_{t_i}, \ \mathcal{L}_{t_i}}(s(t_i)) -
h_{\overline{A},\ \mathcal{L}}(s(t_i))\right| < c_2h_{\overline{S},
\ \epsilon^*(\mathcal{L})}(t_i) + c_3.
\end{eqnarray}
Because $h_{\overline{A},\ \mathcal{L}}(s(t_i)) = h_{\overline{S},\
s^*(\mathcal{L})}(t_i)$ and because both $\epsilon^*(\mathcal{L})$
and $ s^*(\mathcal{L})$ are ample, there exists positive constants
$c_4$ and $c_5$ such that
\begin{eqnarray}\label{comparison of weil height}
h_{\overline{A},\ \mathcal{L}}(s(t_i)) \leq c_4 h_{\overline{S}, \
\epsilon^*(\mathcal{L})}(t_i) + c_5
\end{eqnarray}


As indicated in \cite{Zarhin}, Zarhin's trick works for families and
therefore $B = (A \times A^{\vee})^4$ is an abelian scheme over $S$
with principal polarization. Because a constant family of abelian
varieties contains no nonconstant subfamily, $B$ is also
non-isotrivial. The modular map $J: S \to M_{g, 1}$ attached to $B$
with respect to this principle polarization is nonconstant. Let
$\mathcal{N}$ be an ample line bundle of the Baily-Borel
compactification of $M_{g, 1}$. By another inequality of Faltings
\cite{Fa83}, by $h_{F}((A_t \times A_t^{\vee})^4) = 8 h_F(A_t)$ and
by the fact that $h_{\overline{S}, \ \epsilon^*(\mathcal{L})} $ is
almost proportional to $ h_{\overline{S}, \ J^*(\mathcal{N})}$ there
exist positive constants $c_6$, $c_7$ and $c_8$ such that
\begin{eqnarray}\label{deepest inequality}
|h_F(A_{t_i}) - c_6h_{\overline{S}, \ \epsilon^*(\mathcal{L})}(t_i)|
\leq c_7 +c_8 \log(\max(1, h_{\overline{S}, \
\epsilon^*(\mathcal{L})}(t_i))).
\end{eqnarray}
Notice that although not explicitly mentioned in \cite{Fa83}, one
can check carefully (or see the lemma below) that $c_8$ is
independent of the number field $K$. Combining (\ref{change of
canonical height}), (\ref{SilvermanTate}) and (\ref{comparison of
weil height}) there exist positive constants $c_9$ and $c_{10}$ such
that
\begin{eqnarray}\label{final1}
c_9(\deg{\phi_i})^{1/g} \leq h_{\overline{S}, \
\epsilon^*(\mathcal{L})}(t_i) + c_{10}.
\end{eqnarray}
Combining (\ref{general faltings inequality}) and (\ref{deepest
inequality}) there exist positive constants $c_{11}, c_{12}$  and
$c_{13}$ such that
\begin{eqnarray}\label{final2}
h_{\overline{S}, \ \epsilon^*(\mathcal{L})}(t_i) \leq c_{11} +
c_{12} \log(\max(1, h_{\overline{S}, \
\epsilon^*(\mathcal{L})}(t_i))) + c_{13}\log(\deg{\phi_i}).
\end{eqnarray}
It is clear that (\ref{final1}) contradicts to (\ref{final2}) as the
degree of the isogenies $\deg{\phi_i}$ goes to the infinity.
\end{proofof}

Lastly we sketch a calculation to make sure that positive constants
$c_3, c_4$ obtained in \cite[p.356]{Fa83} are independent of number
fields.
\begin{lemma}
Let $X \subset \PP_{\ZZ}^n$ be Zariski-closed, $Y \subset X$ closed,
$\|\ \|$ a hermitian metric on $\mathcal{O}(1)|(X(\CC) - Y(\CC))$
with logarithmic singularities along $Y$ and $\|\ \|_1$ a hermitian
metric on $\mathcal{O}(1)|X(\CC)$. For $x \in X(\overline{\QQ}) -
Y(\overline{\QQ})$ one defines $h(x)$ and $h_1(x)$ as in
\cite{Fa83}. There exists positive constant $c_3$ and $c_4$ such
that for all $x \in X(\overline{\QQ}) - Y(\overline{\QQ})$ we have
\begin{eqnarray*}
|h(x) - h_1(x)| \leq c_3 + c_4\log(\max(1, h_1(x))).
\end{eqnarray*}
\end{lemma}
\begin{proof}
We may assume that $Y$ is the intersection of $X$ with a linear
subspace and the set of common zeros of $f_1, \ldots, f_r \in
\Gamma(X/\ZZ, \mathcal{O}(1))$ is exactly $Y$. By multiplying the
metric we assume $||f_i||_1 \leq 1$. A rational point $x \in X(K) -
Y(K)$ corresponds to $\rho: \mathrm{Spec}(R) \to X$, where $R$ is
the integer ring of the number field $K$.  We assume $f_1(x) \neq
0$. By definition we have
\begin{eqnarray*}
[K : \QQ] h_1(x) \ge \sum_{\sigma} - \log||f_1||_1{(\sigma(x))}
\end{eqnarray*}
and $ [K: \QQ]|h(x) - h_1(x)|= \big|\sum_{\sigma}
\log\left(||f_1||/||f_1||_1\right)(\sigma(x))\big|,$ where $\sigma$
runs through all embeddings of $K \hookrightarrow \CC$. According to
logarithmic singularities of the metric there exist positive
constants $c_1, c_2$ such that for all number field $K$ and for $K$
rational point $x$
\begin{eqnarray*}
[K: \QQ]|h(x) - h_1(x)| \leq c_1[K: \QQ] + c_2 \sum_{\sigma}\log
\left(-\log||f_1||_1(\sigma(x))\right).
\end{eqnarray*}
Furthermore we have
\begin{eqnarray*}
\sum_{\sigma} \log \big(- \log||f_1||_1(\sigma(x))\big) &\leq&
\log\big(\sum_{\sigma} - \log||f_1||_1{(\sigma(x))}/[K:
\QQ]\big)^{[K : \QQ]} \\ &\leq& [K: \QQ]\log h_1(x).
\end{eqnarray*}
The above inequalities prove the desired claim.
\end{proof}

\

\begin{Acknowledgements}
\noindent Ming-Xi Wang would like to thank P. Habegger for
introducing him this subject and for helpful conversations.  We
thank C. Fuchs and G. R\'emond for their comments.
\end{Acknowledgements}


\bigskip
\noindent

\end{document}